\documentclass[12pt]{amsart}
\usepackage{a4wide}
\usepackage{amsmath,amssymb,amsthm}
\usepackage{amsfonts,amsthm,amsmath,amssymb,amscd,color}
\usepackage{graphicx}
\usepackage{ulem}
\usepackage{enumerate}
\usepackage{parskip}
\usepackage{graphicx}
\usepackage{verbatim} 

\newcommand{\RR}{\mathbb{R}}

\newcommand{\NN}{\mathbb{N}}

\newcommand{\eps}{\varepsilon}

\newcommand{\CA}{\mathcal{A}}
\newcommand{\CF}{\mathcal{F}}

\newcommand{\CP}{\mathcal{P}}
\newcommand{\CalC}{\mathcal{C}}
\newcommand{\CM}{{\mathcal{M}}}
\newcommand{\CE}{{\mathcal{E}}}
\newcommand{\CDD}{{\mathcal{D}}}

\newcommand{\BL}{\mathrm{BL}}
\newcommand{\TV}{\mathrm{TV}}

\newtheorem{thrm}{Theorem}[section]
\newtheorem{prop}[thrm]{Proposition}
\newtheorem{lemma}[thrm]{Lemma}
\newtheorem{clry}[thrm]{Corollary}

\newtheorem{example}[thrm]{Example}
\newtheorem{remark}[thrm]{Remark}

\newcommand{\smfrac}[2]{\mbox{$\frac{#1}{#2}$}}
\newcommand{\pair}[2]{\left\langle #1 , #2 \right\rangle}

\newcommand{\n}{{(n)}}
\newcommand{\cco}{{\overline{\mathrm{co}}}}

\DeclareMathOperator{\supp}{supp}

\newcommand{\Itr}{\textup{Int}}

\def\d{{\rm d}}
\def\Lip{{\rm Lip}}
\def\Int{{\rm Int}}

\begin{document}

\title[Equicontinuous families of Markov operators]{Equicontinuous families of Markov operators in view of asymptotic stability}
\author{Sander C. Hille}
\address{Mathematical Institute, Leiden University, P.O. Box 9512, 2300 RA Leiden, The Netherlands, (SH,MZ)}
\email{\{shille,m.a.ziemlanska\}@math.leidenuniv.nl}
\author{Tomasz Szarek}
\address{Institute of Mathematics, University of Gda\'nsk, Wita Stwosza
  57, 80-952 Gda\'nsk, Poland, (TS)}
\email{szarek@intertele.pl}
\author{Maria A. Ziemla\'nska}

\date\today

\thanks{The work of Maja Ziemla\'nska has been partially supported by a Huygens Fellowship of Leiden University. The work of Tomasz Szarek has been supported by the National Science Centre
of Poland, grant number 2016/21/B/ST1/00033.}

\maketitle


\begin{abstract}
Relation between equicontinuity -- the so-called e--property and stability of Markov operators is studied. In particular, it is shown that any asymptotically stable Markov operator with an invariant measure such that the interior of its support is non-empty satisfies the e--property.
\end{abstract}


\section{Introduction}

This paper is centered around two concepts of equicontinuity for Markov operators defined on probability measures on Polish spaces: the {\it e-property} and the {\it Ces\`aro e-property}. Both appeared as a condition (among others) in the study of ergodicity of Markov operators. In particular they are very useful in proving existence of a unique invariant measure and its {\it asymptotic stability}: at whatever probability measure one starts, the iterates under the Markov operator will weakly converge to the invariant measure. The first concept appeared in \cite{Lasota-Szarek:2006,Szarek-Worm:2012} while the second was introduced  in \cite{Worm-thesis:2010} as a theoretical generalisation of the first and allowed the author to extend various results by replacing the e-property condition by apparently weaker the Ces\`aro e-property condition, among others. 


Interest in equicontinuous families of Markov operators existed already before the introduction of the e-property. Jamison \cite{Jamison:1964}, working on compact metric state spaces, introduced the concepts of (dual) Markov operators on the continuous functions that are `uniformly stable' or `uniformly stable in mean' to obtain a kind of asymptotic stability results in this setting. Meyn and Tweedie \cite{Meyn-Tweedie:2009} introduced the so-called `{\it e-chains}' on locally compact Hausdorff topological state spaces, for similar purposes. See also \cite{Zaharopol:2014} for results in a locally compact metric setting.

The above mentioned concepts were used in proving ergodicity for some Markov chains (see \cite{Stettner: 1994, Czapla: 2012, CzaplaHorbacz: 2014, EsS-vonRenesse: 2012, GongLiu: 2015, Komorowksi_ea:2010, WW: 2017}).

It is worth mentioning here that similar concepts appear in  the study of mean equicontinous dynamical systems mainly on compact spaces (see for instance \cite{LiTuJe: 2015}). However it must be stressed here that our space of Borel probability measures defined on some Polish space is non-compact.

Studing the e--property the natural question arose whether any asymptotically stable Markov operator satisfies this property. Proposition 6.4.2 in \cite{Meyn-Tweedie:2009} says that this holds when the phase space is compact. In particular, the authors showed that the stronger e--chain property is satisfied. Unfortunately, the proof contains a gap and it is quite easy to construct an example showing that some additional assumptions must be then added. 

On the other hand, striving to repair the gap of the Meyn-Tweedie result mentioned above, we show that any asymptotically stable Markov operator with an invariant measure such that the interior of its support is nonempty satisfies the e--property.

\section{Preliminaries}

Let $(S, d)$ be a Polish space. By $B(x,r)$ we denote the open ball in $(S,d)$ of radius $r$, centered at $x\in S$ and $\partial B(x,r)$ denotes its boundary. Further $\overline{E}$, $\Int_S E$ denote the closure of $E\subset S$ and the interior of $E$, respectively. By $C_b(S)$ we denote the vector space of all bounded real-valued continuous functions  on $S$ and by $B_b(S)$ all bounded real-valued Borel measurable functions, both equipped with the supremum norm $|\cdot|$. By $L_b(S)$  we denote the subspace of $C_b(S)$ of all bounded Lipschitz functions (for the metric $d$ on $S$). For $f\in L_b(S)$, $\Lip f$ denotes the Lipschitz constant of $f$.

By $\CM(S)$  we denote the family of all finite Borel measures on $S$ and  by $\CP(S)$ the subfamily of all probability measures in $\CM (S)$. For $\mu\in\CM (S)$, its {\it support} is the set
\[
\supp\mu :=\{x\in S:  \mu(B(x,r))>0\,\,\mbox{for all}\ r>0\}.
\]
An operator $P: \CM (S)\to \CM (S)$ is called a {\it Markov operator} (on $S$) if it satisfies the following two conditions:
\begin{enumerate}
\item[({\it i})] {\it (Positive linearity)} $P(\lambda_1\mu_1+\lambda_2\mu_2)=\lambda_1 P\mu_1+\lambda_2 P\mu_2$\\
for $\lambda_1, \lambda_2\geq 0$; $\mu_1, \mu_2\in\CM(S)$;
\item[({\it ii})] {\it (Preservation of the norm)} $P\mu (S)=\mu (S)$ for $\mu\in\CM(S)$.
\end{enumerate}
A measure $\mu_*$ is called {\it invariant} if $P\mu_*=\mu_*$. A Markov operator $P$ is {\it asymptotically stable} if there exists a unique invariant measure $\mu_*\in\CP(S)$ such that $P^n\mu\to\mu_*$ weakly as $n\to\infty$ for every $\mu\in\CP(S)$.

For brevity we shall use the notation:
$$
\langle f, \mu\rangle :=\int_S f(x)\mu (\d x)\qquad\mbox{for}\ f\in B_b(S),\ \mu\in\CM (S).
$$

A Markov operator $P$ is {\it regular} if there exists a linear operator $U: B_b(S)\to B_b(S)$ such that 
\[
\pair{f}{P\mu} = \pair{Uf}{\mu}\qquad \mbox{for all}\ f\in B_b(S),\ \mu\in\CM(S).
\]
The operator U is called the {\it dual operator} of $P$. A regular Markov operator is a {\it Feller operator} if its dual operator $U$ maps $C_b(S)$ into itself.
Equivalently, $P$ is Feller if it is continuous in the weak topology (cf. \cite{Worm-thesis:2010}, Proposition 3.2.2). 

A Feller operator $P$ satisfies the {\it e--property} at $z\in S$ if for any $f\in L_b(S)$ we have 
\begin{equation}\label{e2_9.04.17}
\lim_{x\to z} \limsup_{n\to\infty}|U^n f (x) - U^n f (z)|=0,
\end{equation}
i.e., if the family of iterates $\{U^n f: n\in\mathbb N\}$ is equicontinuous at $z\in S$. We say that a Feller operator satisfies the {\it e--property} if it satisfies it at any $z\in S$.

D. Worm slightly generalized the e--property introducing the Ces\'aro e--property (see  \cite{Worm-thesis:2010}). Namely,  a Feller operator $P$ will satisfy the {\it Ces\'aro e--property} at $z\in S$ if for any $f\in L_b (S)$ we have 
\begin{equation}\label{e3_9.04.17}
\lim_{x\to z} \limsup_{n\to\infty}\left|\frac{1}{n}\sum_{k=1}^n U^k f (x) - \frac{1}{n}\sum_{k=1}^n U^k f (z)\right|=0.
\end{equation}
Analogously a Feller operator satisfies the {\it Ces\'aro e--property} if it satisfies this property at any $z\in S$.

The following simple example shows that Proposition 6.4.2 in \cite{Meyn-Tweedie:2009} fails.

\begin{example}
Let $S=\{1/n : n\ge 1\}\cup\{0\}$ and let $T: S\to S$ be given by the following formula:
$$
T(0)=T(1)=0\qquad\text{and}\qquad T(1/n)=1/(n-1)\quad\text{for $n\ge 2$}.
$$
The operator $P: \mathcal M(S)\to\mathcal M(S)$ given by the formula $P\mu=T_*(\mu)$ (the pushforward measure) is asymptotically stable but it does not satisfy the e--property at $0$.
\end{example}

Jamison \cite{Jamison:1964} introduced for a Markov operator the property of {\it uniform stability in mean} when $\{U^n f:n\in\mathbb N\}$ is an equicontinuous family of functions in the space of real-vauled continuous function $C(S)$ for every $f\in C(S)$. Here $C$ is a compact metric space. Since the space of bounded Lipschitz functions is dense for the uniform norm in the space of bounded uniform continuous functions, this property coincides with the Ces\'aro e--property for compact metric spaces. Now, if the Markov operator $P$ on the compact metric space is asymptotically stable, with the invariant measure $\mu_*\in\mathcal M_1$, then  
$
\frac{1}{n}\sum_{i=1}^n U^i f\to \langle f, \mu_*\rangle$ pointwise, for every $f\in C(S)$. According to Theorem 2.3 in \cite{Jamison:1964} this implies that $P$ is uniformly stable in mean, i.e., has 
the Ces\'aro e--property.

\begin{example} Let $(k_n)_{n\ge 1}$ be an increasing sequence of prime numbers. Set
$$
S:=\{(\overbrace{\mathstrut 0, \ldots, 0,}^{k_n^i-1 - times} i/k_n, 0,\ldots )\in l^{\infty}: i\in\{0,\ldots, k_n\}, n\in\mathbb N\}.
$$ 
The set $S$ endowed with the $l^{\infty}$-norm $\|\cdot\|_{\infty}$ is a (noncompact) Polish space. Define $T: S\to S$ by the formula
$$
T((0,\ldots ))=T((\overbrace{\mathstrut 0, \ldots, 0,}^{k_n^{k_n}-1- times} 1, 0,\ldots))=(0,\ldots, 0, \ldots)\qquad\text{for $n\in\mathbb N$}
$$
and
$$
T((\overbrace{\mathstrut 0, \ldots, 0,}^{k_n^i-1- times} i/k_n,  0,\ldots ))=(\overbrace{\mathstrut 0, \ldots, 0,}^{k_n^{i+1}-1- times} (i+1)/k_n, 0,\ldots )\quad\text{for $i\in\{1,\ldots, k_n-1\}, n\in\mathbb N$.}
$$
The operator $P: \mathcal M(S)\to\mathcal M(S)$ given by the formula $P\mu=T_*(\mu)$ is asymptotically stable but it does not satisfy the Ces\'aro  e--property at $0$. Indeed, if we take an arbitrary continuous function $f: S\to\mathbb R_+$ such that $f((0,\ldots, 0, \ldots))=0$ and $f(x)=1$ for $x\in S$ such that $\|x\|_{{\infty}}\ge1/2$ we have
$$
\frac{1}{k_n}\sum_{i=1}^{k_n} U^i f((\overbrace{\mathstrut 0, \ldots, 0}^{k_n-1}, 1/k_n, 0,\ldots ))-\frac{1}{k_n}\sum_{i=1}^{k_n} U^i f((0,\ldots))\ge 1/2.
$$
\end{example} 

We are in a position to formulate the main result of our paper:

\begin{thrm}\label{main_theorem} Let $P$ be an asymptotically stable Feller operator and let $\mu_*$ be its unique invariant measure.
If $\Int_S(\supp\mu_*)\neq\emptyset$, then $P$ satisfies the e--property.
\end{thrm}

Its proof involves the following lemma:
\begin{lemma}Let $P$ be an asymptotically stable Feller operator and let $\mu_*$ be its unique invariant measure. Let $U$ be dual to $P$. If $\Itr_S(\supp\mu_*)\neq\emptyset$, then for every $f\in C_b(S)$ and any $\varepsilon>0$ there exists a ball $B\subset\supp\mu_*$ and $N\in\NN$ such that
\begin{equation}\label{e1_19.11.16}
|U^nf(x)-U^n f(y)|\le\varepsilon\qquad\text{for any $x, y\in B$,\ $n\geq N$}.
\end{equation}
\end{lemma}

\begin{proof} Fix $f\in C_b(S)$ and $\varepsilon>0$. Let $W$ be an open set such that $W\subset\supp\mu_*$. Set $Y={\overline{W}}$ and observe that the subspace $Y$ is a Baire space. Set
$$
Y_n:=\{x\in Y: |U^mf(x)-\langle f, \mu_*\rangle | \le\varepsilon/2\,\,\, \text{for all $m\ge n$}\}
$$
and observe that $Y_n$ is closed and
$$
Y=\bigcup_{n=1}^{\infty} Y_n.
$$
By the Baire category theorem there exist $N\in\mathbb N$ such that $\Itr_Y Y_{N}\neq\emptyset.$
Thus there exists a set $V\subset Y_N$ open in the space $Y$ and consequently a ball  $B$ in $S$ such that
$B\subset Y_N\subset\supp\mu_*$. Since 
$$
|U^nf(x)-\langle f, \mu_*\rangle | \le\varepsilon/2\qquad\text{for any $x\in B$ and $n\ge N$},
$$
condition (\ref{e1_19.11.16}) is satisfied.
\end{proof}

We are ready to prove Theorem \ref{main_theorem}.
\begin{proof} Assume, contrary to our claim, that $P$ does not satisfy the e--property. Therefore there exist a function $f\in C_b(S)$ and a point $x_0\in S$ such that
$$
\limsup_{x\to x_0} \limsup_{n\to\infty}|U^n f(x)-U^n f(x_0)|>0.
$$
Choose $\varepsilon>0$ such that
$$
\limsup_{x\to x_0} \limsup_{n\to\infty}|U^n f(x)-U^n f(x_0)|\ge 3\varepsilon.
$$
Let $B:=B(z, 2r)$ be a ball such that condition (\ref{e1_19.11.16}) holds. Since $B(z, r)\subset\supp\mu_*$, we have $\gamma:=\mu_*(B(z, r))>0$. Choose $\alpha\in (0, \gamma)$. Since the operator $P$ is asymptotically stable, we have
\begin{equation}\label{e2_19.11.16}
\liminf_{n\to\infty}P^n\mu (B(z, r))>\alpha\qquad\text{for all $\mu\in\CP(S)$},
\end{equation}
by the Alexandrov theorem (see \cite{Ethier_Kurtz}).

Let $k\ge 1$ be such that $2(1-\alpha)^k|f|<\varepsilon$. By induction we are going to define two sequences of measures $(\nu^{x_0}_i)_{i=1}^k$, 
$(\mu^{x_0}_i)_{i=1}^k$ and a sequence of integers $(n_i)_{i=1}^k$ in the following way: let $n_1\ge 1$ be such that
\begin{equation}
P^{n_1}\delta_{x_0}(B(z, r))>\alpha. 
\end{equation}
Choose $r_1<r$ such that $ P^{n_1}\delta_{x_0}(B(z, r_1))>\alpha$ and $P^{n_1}\delta_{x_0}(\partial B(z, r_1))=0$ and set 
\begin{equation}
\nu^{x_0}_1(\cdot)=\frac{P^{n_1}\delta_{x_0}(\cdot\cap B(z, r_1))}{P^{n_1}\delta_{x_0}(B(z, r_1))}
\end{equation}
and
\begin{equation}
\mu^{x_0}_1(\cdot)=\frac{1}{1-\alpha}\left(P^{n_1}
\delta_{x_0}(\cdot)-\alpha\nu_1^{x_0} (\cdot)\right).
\end{equation}
Assume that we have done it for $i=1,\ldots, l$, for some $l<k$. Now let $n_{l+1}$ be such that
\begin{equation}
P^{n_{l+1}}\mu_l^{x_0}(B(z, r))>\alpha. 
\end{equation}
Choose $r_{l+1}<r$ such that $ P^{n_{l+1}}\mu_l^{x_0}(B(z, r_{l+1}))>\alpha$ and $P^{n_{l+1}}\mu_l^{x_0}(\partial B(z, r_{l+1}))=0$ and set 
\begin{equation}
\nu^{x_0}_{l+1}(\cdot)=\frac{P^{n_{l+1}}\mu_l^{x_0}(\cdot\cap B(z, r_{l+1}))}{P^{n_{l+1}}\mu_l^{x_0}(B(z, r_{l+1}))}
\end{equation}
and
\begin{equation}
\mu^{x_0}_{l+1}(\cdot)=\frac{1}{1-\alpha}\left(P^{n_{l+1}}
\mu^{x_0}_{l}(\cdot)-\alpha\nu_{l+1}^{x_0} (\cdot)\right).
\end{equation}
We are done. We have
$$
\begin{aligned}
P^{n_1+\ldots+n_k}\delta_{x_0}(\cdot)&=\alpha P^{n_2+\ldots+n_k}\nu^{x_0}_1(\cdot)
+\alpha (1- \alpha)
P^{n_3+\ldots+n_k}\nu^{x_0}_2(\cdot)
+\ldots + \\
&+\alpha (1- \alpha )^{k-1}
\nu^{x_0}_k (\cdot)+(1-\alpha)^k\mu^{x_0}_k (\cdot).
\end{aligned}
$$
By induction we check that $\nu_i^x-\nu_i^{x_0}\to 0$ and $\mu_i^x-\mu_i^{x_0}\to 0$ weakly as $d(x, x_0)\to 0$.
Indeed, if $i=1$, then $\nu_1^x-\nu_1^{x_0}\to 0$ weakly (as $d(x, x_0)\to 0$), by the fact that $P$ is a Feller operator and $\lim_{d(x, x_0)\to 0}P^{n_1}\delta_{x}(B(z, r_1))=P^{n_1}\delta_{x_0}(B(z, r_1))$, by the Alexandrov theorem due to the fact that $P^{n_1}\delta_{x_0}(\partial B(z, r_1))=0$. On the other hand, the weak convergence $\mu_1^x-\mu_1^{x_0}\to 0$ as $d(x, x_0)\to 0$ follows directly from the definition of $\mu_1^x$. Moreover, observe that for $x$ sufficiently close to $x_0$ we have $P^{n_1}\delta_{x}(B(z, r))>\alpha$ and therefore $\mu_1^x\in\CP(S)$.

Assume now that we have proved that $\nu_i^x-\nu_i^{x_0}\to 0$ and $\mu_i^x-\mu_i^{x_0}\to 0$ weakly as $d(x, x_0)\to 0$ for $i=1,\ldots, l$. We show that $\nu_{l+1}^x-\nu_{l+1}^{x_0}\to 0$ and $\mu_{l+1}^x-\mu_{l+1}^{x_0}\to 0$ weakly as $d(x, x_0)\to 0$ too. Analogously, $\lim_{d(x, x_0)\to 0}P^{n_{l+1}}\mu_{l}^{x}(B(z, r_{l+1}))=P^{n_{l+1}}\mu_{l}^{x_0}(B(z, r_{l+1}))$, by the Alexandrov theorem due to the fact that $P^{n_{l+1}}\mu_{l}^{x_0}(\partial B(z, r_{l+1}))=0$ and from the definition of $\nu_{l+1}^x$ we obtain that $\nu_{l+1}^x-\nu_{l+1}^{x_0}\to 0$ weakly as $d(x, x_0)\to 0$. The weak convergence $\mu_{l+1}^x-\mu_{\l+1}^{x_0}\to 0$ as $d(x, x_0)\to 0$ follows now directly from the definition of $\mu_{l+1}^x$ and for $x$ sufficiently close to $x_0$ we have $P^{n_{l+1}}\mu_{l}^{x}(B(z, r))>\alpha$ and therefore $\mu_{l+1}^x\in\CP(S)$. We are done.

Observe that for any $x$ sufficiently close to $x_0$ and  all $n\ge n_1+\ldots+n_k$ we have
$$
\begin{aligned}
P^{n}\delta_{x}(\cdot)&=\alpha P^{n-n_1}\nu^{x}_1(\cdot)
+\alpha (1- \alpha)
P^{n-n_1-n_2}\nu^{x}_2(\cdot)
+\ldots \\
&+\alpha (1- \alpha )^{k-1}
P^{n-n_1-\ldots-n_k}\nu^{x}_k (\cdot)+(1-\alpha)^k P^{n-n_1-\ldots-n_k}\mu^{x}_k (\cdot),
\end{aligned}
$$
where $\supp\nu^x_i\subset B(z, r)$ for all $i=1,\ldots, k$. Thus
\begin{align}
\limsup_{n\to\infty}|\langle f, P^n\nu^x_i\rangle-\langle f, P^n\nu^{x_0}_i\rangle|&=\limsup_{n\to\infty} |\langle U^nf-\langle f, \mu_*\rangle, \nu^x_i\rangle-\langle U^nf -\langle f, \mu_*\rangle, \nu^{x_0}_i\rangle|\nonumber\\
&\le \varepsilon/2+\varepsilon/2=\varepsilon\label{e_20.11.16} 
\end{align}
for all $i=1,\ldots, k$ and $x$ sufficiently close to $x_0$. Hence
$$
\begin{aligned}
3\varepsilon<\limsup_{x\to x_0} \limsup_{n\to\infty}&|U^n f(x)-U^n f(x_0)|=
\limsup_{x\to x_0} \limsup_{n\to\infty}|\langle f, P^n\delta_x\rangle - \langle f, P^n\delta_{x_0}\rangle|\\
&\le \varepsilon (\alpha+\alpha (1-\alpha)+\ldots \alpha (1-\alpha)^{k-1}) +2(1-\alpha)^k |f|\\
&\le \varepsilon+\varepsilon=2\varepsilon,
\end{aligned} 
$$
which is impossible. This completes the proof.
\end{proof}
{\bf Acknowledgements.}\ We thank Klaudiusz Czudek for providing us with Example 1.

\end{document}